\documentclass[11pt]{amsart}
\usepackage{amscd}
\usepackage{amssymb}

\theoremstyle{plain} 
\newtheorem{theorem}{Theorem}[section]
\newtheorem{proposition}[theorem]{Proposition}
\newtheorem{lemma}[theorem]{Lemma}
\newtheorem{corollary}[theorem]{Corollary} 
\newtheorem{question}[theorem]{Question}

\theoremstyle{remark}
\newtheorem{remark}[theorem]{Remark}
\newtheorem{example}[theorem]{Example}

\theoremstyle{definition}
\newtheorem{definition}[theorem]{Definition}

\DeclareMathOperator{\Rat}{Rat}
\DeclareMathOperator{\Spec}{Spec}
\DeclareMathOperator{\Orb}{Orb}

\DeclareMathOperator{\Gal}{Gal}

\DeclareMathOperator{\Per}{Per}
\DeclareMathOperator{\Frob}{Frob}
\newcommand{\fp}{ {\mathfrak p} }
\newcommand{\fc}{ {\mathfrak c }}
\newcommand{\fm}{ {\mathfrak m} }
\newcommand{\fq}{ {\mathfrak q} }

\newcommand{\fo}{ {\mathfrak o} }

\newcommand{\cO}{ {\mathcal O} }
\newcommand{\cS}{ {\mathcal S} }
\newcommand{\cP}{ {\mathcal P} }
\newcommand{\cU}{ {\mathcal U} }

\newcommand{\lra} {\longrightarrow}
\newcommand{\bZ} { {\mathbb Z}} 
 
\newcommand{\bA} { {\mathbb A}} 
\newcommand{\F} { {\mathbb F}} 
\newcommand{\bP} { {\mathbb P}} 
\newcommand{\bQ} { {\mathbb Q}} 
\newcommand{\bC} { {\mathbb C}}
\newcommand{\tp}{ {\tilde \varphi}}
\DeclareMathOperator{\res}{Res}
\DeclareMathOperator{\FPP}{FPP}
\DeclareMathOperator{\Fix}{Fix}
\DeclareMathOperator{\Norm}{N}

\begin{document}
\title{Wreath products and proportions of periodic points}

\author{J.~Juul}
\address{
Jamie Juul\\
Department of Mathematics\\
University of Rochester\\
Rochester, NY 14627\\
USA
}
\email{jamie.l.rahr@gmail.com}

\author {P.~Kurlberg}
\address{
P\"ar Kurlberg\\
Department of Mathematics\\
KTH\\
SE-100 44 Stockholm\\
Sweden}
\email{kurlberg@math.kth.se}

\author{K.~Madhu}
\address{
Kalyani Madhu\\
Department of Mathematics\\
University of Rochester\\
Rochester, NY 14627\\
USA
}
\email{kalyanikmadhu@gmail.com}

\author{T.~J.~Tucker}
\address{
Thomas Tucker\\
Department of Mathematics\\
University of Rochester\\
Rochester, NY 14627\\
USA
}
\email{tjtucker@gmail.com}

\begin{abstract}
Let $\varphi: \bP^1 \lra \bP^1$ be a rational map of degree greater
than one defined over a number field $k$.  For each prime $\fp$ of
good reduction for $\varphi$, we let $\varphi_\fp$ denote the
reduction of $\varphi$ modulo $\fp$.  A random map heuristic suggests
that for large $\fp$, the proportion of periodic points of
$\varphi_\fp$ in $\bP^1(\fo_k/\fp)$ should be small.  We show that
this is indeed the case for many rational functions $\varphi$.  
\end{abstract}

\maketitle

\section{Introduction}
Let $f$ be a polynomial in $\bZ[x]$ of degree greater than 1.  Then $f$
induces a map $f_p: \F_p \lra \F_p$  for each prime $p$ via reduction
modulo $p$.  
Any point $\alpha \in  \F_p$ will be {\it preperiodic} under $f_p$; the
fact that $\F_p$ is finite means that there must be some $i \not= j$
such that $f_p^i(\alpha) = f_p^j(\alpha)$. On the other hand, $\alpha$
may not be \textit{periodic}, since it is quite possible that there is no $n >
0$ such that $f_p^n(\alpha) = \alpha$.  

The model of random maps, along with the heuristic of the birthday
problem, suggests that for a typical $\alpha$ and a typical $f_p$, the
size of the orbit $\Orb_{f_p}(\alpha) = \{ \alpha, f(\alpha), \dots,
f^m(\alpha), \dots \}$ will be about $\sqrt{p}$ (see \cite{FO, Bach,
  Silver, BGH}).  Hence, one might guess that there is about a
$1/\sqrt{p}$ chance that a given $\alpha$ is $f_p$-periodic, and that
the proportion of $f_p$-periodic points in $\F_p$ is about $1/\sqrt{p}$.
In particular, one would then have 
\[ \lim_{p \to \infty} \frac{\#\Per(f_p)}{p} = 0, \] where $\Per(f_p)$
is the set of points in $\F_p$ that are $f_p$-periodic.  

More generally, one might consider this problem for rational functions
over number fields.  Let $k$ be a number field and let $\varphi \in
k(x)$ be a rational function of degree greater than one.  For all but
finitely many primes $\fp$ in the ring of integers $\fo_k$ of $k$,
reducing modulo $\fp$ gives rise to a well-defined map $\varphi_\fp: \bP^1(\fo_k /
\fp) \lra   \bP^1(\fo_k /
\fp)$.  We let $\Norm(\fp)$ denote the number of elements in the
residue field $\fo_k / \fp$.  Then one might expect for a typical
$\varphi$, taking the limit over the $\fp$ such that $\varphi_\fp$ is a
 well-defined map on $(\fo_k /
\fp)$, one should have
\begin{equation}\label{eq1}  \lim_{\Norm(\fp) \to \infty} \frac{\#\Per(\varphi_\fp)}{\Norm(\fp)
  + 1} = 0.
\end{equation}

Of course, this might not necessarily be the case.  For example, if
$f(x)$ is a powering map $f(x) = x^n$, then $f_p$ is a bijection for
all $p \not\equiv 1 \pmod{n}$ and thus all points in $\F_p$ are
$f_p$-periodic for all $p \not\equiv 1 \pmod{n}$.  A more general
family of examples comes from Dickson polynomials, which are defined
by $f(x + a/x) = x^n + (a/x)^n$ (when $a=0$, one has a powering map).
Fried \cite{Fried} showed that if $f$ is any polynomial over a number
field $k$ such that $f_\fp$ is a bijection for infinitely many primes
$\fp$ in $\fo_k$, then $f$ can be a written as a composition of
Dickson polynomials and linear polynomials (polynomials of the form
$ax + b)$.  More recently, Guralnick, M\"uller, and Saxl \cite{GMS}
have given a classification of all indecomposable rational functions
$\varphi$ over number fields such that $\varphi_\fp$ is a bijection
for infinitely many primes $\fp$; the classification is substantially
more complicated.  The rational functions classified by Guralnick,
M\"uller, and Saxl are often referred to as indecomposable {\it exceptional rational
  functions} (for a more general discussion of exceptional maps, see
\cite{GTZ}).

\begin{question}\label{q1}
  Let $k$ be a number field.  Can one classify all rational functions
  $\varphi \in k(x)$ of degree greater than one over a number field
  $k$ such that \eqref{eq1} fails to hold?  
\end{question}

It is possible that all rational functions such that \eqref{eq1} fails
to hold come from exceptional rational functions, but we are not able
to prove it at the present time. However, we have some evidence that
this may be the case.  We can show that for ``most'' rational functions
$\varphi$ of degree $d$, the proportion
$\#\Per(\varphi_\fp)/(\Norm(\fp) + 1)$ becomes small for large
$\Norm(\fp)$. To phrase this precisely, we need a bit more notation.

Let $k$ be a number field.  Given a point $(a_0,\dots, a_d, b_0,
\dots, b_d)$ in $\bA^{2d+2}(\bar{k})$, we set ${\vec a} = (a_0,\dots, a_d)$, ${\vec
  b} = (b_0,\dots,b_d)$, $p_{\vec a} = a_d x^d + \dots + a_0$,
$q_{\vec b} = b_d x^d + \dots + b_0$, and $\varphi_{ {\vec a}, {\vec
    b}} = p_{\vec a}/ q_{\vec b}$.  If the resultant of $p$ and $q$ is
nonzero and either $a_d$ or $b_d$ is nonzero, then $\varphi_{ {\vec
    a}, {\vec b}} = p_{\vec a}/ q_{\vec b}$ is a rational function of
degree $d$ in $k(x)$.  We denote the set of such $({\vec a}, {\vec b})$ as
$\Rat_d$. 
 
\begin{theorem}\label{generic}
Let $\epsilon > 0$ and $d > 1$.  With notation as above, there is a
Zariski dense open subset $U_{d,\epsilon}$ of $\Rat_d$ such that for any number field
$k$ and any $({\vec a},
{\vec b}) \in \Rat_d$, we have 
\[ \lim_{\substack{\Norm(\fp) \to \infty \\ \text{primes $\fp$ of
    $\fo_k$}}}  \frac{\#\Per(\varphi_{{\vec a}, {\vec b}})}{\Norm(\fp) +
  1} \leq \epsilon.\]  
\end{theorem}

We are also able to prove the following.  

\begin{theorem}\label{collide}
Let $\varphi$ be a rational function of degree $d > 1$ such that for
any two distinct critical points $\alpha_1, \alpha_2$ of $\varphi$ and
any positive integers $m$ and $n$, we have $\varphi^m(\alpha_1) \not=
\varphi^n(\alpha_2)$ unless $m=n$ and $\alpha_1 = \alpha_2$.  
Then 
\begin{enumerate}
\item[(a)] \[ \liminf_{\fp \to \infty} \frac{\#\Per(f_\fp)}{\Norm(\fp) + 1}
  = 0;\]
\item[(b)] if $k$ is algebraically closed in $K_1$, we have
 \[\lim_{\fp \to \infty} \frac{\#\Per(f_\fp)}{\Norm(\fp) + 1}
  = 0.\]
\end{enumerate}
\end{theorem}

Theorem~\ref{collide} thus shows that there are essentially only two
obstacles to showing that \eqref{eq1} holds for a given rational
function $\varphi$: (i) intersections between the orbits of the
critical points of $\varphi$ and (ii) nontrivial algebraic extensions
of the ground field $k$ occurring in the splitting field for
$\varphi(x) - t$ over $k(t)$.   To some extent, one can overcome the second problem by
passing to an extension of $k$ and asking instead that
\begin{equation}\label{eq2}  \liminf_{\Norm(\fp) \to \infty} \frac{\#\Per(\varphi_\fp)}{\Norm(\fp)
  + 1} = 0.
\end{equation}

\begin{question}\label{q2}
  Let $k$ be a number field.  Can one classify all rational functions
  $\varphi \in k(x)$ of degree greater than one over a number field
  $k$ such that \eqref{eq2} fails to hold?  
\end{question}

One interesting fact is that \eqref{eq2} holds for powering maps but not
for all Dickson polynomials.  
While the powering map $f(x) = x^n$ induces a bijection $f_p: \F_p
\lra \F_p$ when $p \not\equiv 1 \pmod{n}$, it is easy to see that when
$p \equiv 1 \pmod {n^r}$, we have $\frac{\#\Per(f_p)}{p} \leq
\frac{1}{n^r}+\frac{1}{p}$.  Thus, in this case, we have $\lim\inf_{p \to \infty}
\#\Per(f_p)/p = 0$ (see Example \ref{powering}).  However, as we shall
see in Example \ref{Cheby}, when $f$ is the Dickson polynomial
$f(x+1/x) = x^\ell + (1/x)^\ell$ where $\ell$ is an odd prime, we
have $$\liminf_{p \to \infty} \#\Per(f_p)/p = 1/2,$$ and where
$f(x+1/x) = x^\ell + (1/x)^\ell$ and $\ell = 2$, then $$\liminf_{p
  \to \infty} \#\Per(f_p)/p = 1/4.$$ (Dickson polynomials of the form
$f(x + 1/x) = x^n + (1/x)^n$ are called Chebyshev polynomials).

In the case of quadratic polynomials, Chebyshev polynomials and their
conjugates are the only polynomials such that \eqref{eq2} fails to
hold.  

\begin{theorem}\label{quad}
Let $k$ be a number field and let $f \in k[x]$ be a quadratic
polynomial.  Then 
\[ \liminf_{\fp \to \infty} \frac{\#\Per(f_\fp)}{\Norm(\fp) + 1} = 0\]
unless there is a linear
polynomial $\sigma = ax + b \in k[x]$ such that $\sigma^{-1} f \sigma$
is equal to the Chebyshev polynomial $x^2 - 2$.  
\end{theorem}

More generally, we are able to treat Question \ref{q2} for all maps of
the form $f(x) = x^d + c$ (see Theorem~\ref{uni2}).  The fact that
such maps can be treated is perhaps not surprising in light of related
results of \cite{HJM} and \cite{Jones2}.

Our approach follows that of Odoni \cite{odoni}, though with some
differences.  To describe things better, we need a definition.

\begin{definition} If $H$ is a group acting on a set $S$ then we define $\FPP(H)$ to be the proportion of elements of $H$ fixing some $s\in S$. 
\end{definition}

Let $\psi$ be a rational function defined over $\F_q$, let $K_n$ be the 
splitting field of $\psi^n(x) - t$ over $\F_q(t)$, and let $G_n =
\Gal(K_n / \F_q(t))$; suppose that $\F_q$ is algebraically closed in
$K_n$.  Since $G_n$ acts on
 the set of roots of $\psi^n(x) - t$, it makes sense to consider $\FPP(G_n)$. The Chebotarev density theorem for function fields, due to
Murty and Scherk \cite{murty},  implies
that when $\FPP(G_n)$ is small, then the image of $\bP^1(\F_q)$
under $\psi^n$ is small provided that $q$ is sufficiently large. 
Since a periodic point is in the image of $\bP^1(\F_q)$ under
$\psi^n$ for every $n$ (see Lemma \ref{S}), this means that $\psi$ has few
periodic points.   

We will apply this idea to $\psi$ arising from the reduction
$\varphi_\fp$ of a rational function over a number field $k$ modulo a
prime $\fp$ in $k$.  We will see, via Proposition \ref{EGA}, that for all
but finitely many primes $\fp$, the Galois groups of the splitting fields
of $\varphi_\fp^n(x) - t$ is the same as the Galois groups of the
splitting fields of $\varphi^n(x) - t$ over $k(t)$; let us call this
group $G_n$, as above.  Then it suffices to show that $\FPP(G_n)$ is
very small.  This can be difficult to do in general, but Odoni
\cite[Lemma 4.3]{odoni} has shown that if $G_n$ is the $n$-fold wreath
product $[G]^n$ (see Section \ref{main2}) of some transitive group $G$, then $\lim_{n \to
  \infty} \FPP(G_n) = 0$.

We now give a brief outline of the paper.  After some preliminaries in
Section~\ref{prelim}, we state and prove Theorem \ref{main theorem},
which gives conditions guaranteeing that $G_n = [G]^n$. A key fact
here is that primes in the critical orbit ramify ``disjointly'' in the
sequence of splitting fields of $\varphi^n(x) - t$. That is, for each
$n$ we can find primes that are unramified in the splitting field of
$\varphi^{n-1}(x) - t$ and, in each subextension of the splitting
field of $\varphi^n(x) - t$ over $\varphi^{n-1}(x) - t$, at least one
such prime ramifies that ramifies in no other subextension.  Following
that, we show that Galois groups stay the same after almost all
specializations, provided that the extensions are geometrically
integral, in Section \ref{special}.  Next, in Section \ref{CS}, we use
the Murty-Scherck effective Chebotarev theorem \cite{murty} to bound
proportions of periodic points by proportions of fixed point elements
of Galois groups.  We are then able to prove our main theorems on
proportions of periodic points in Section \ref{main2}.  We conclude
with an elementary discussion of periodic points of powering map,
Chebyshev maps, and Latt\`es maps.  

We note that many of the results in this paper, Theorem \ref{main
  theorem} in particular, should generalize to higher dimensional
situations.  We plan to treat the case of higher dimensions in a
future paper.

\vskip2mm
\noindent {\it Acknowledgments.} We would like to thank Rafe Jones and
Michelle Manes for useful conversations.  The first and fourth authors
were partially supported by NSF grant DMS-1200749.  The second author
was was partially supported by  the G\"oran Gustafsson
Foundation for Research in Natural Sciences and Medicine, and the
Swedish Research Council (621-2011-5498.)

\section{Preliminaries}\label{prelim}

We say that $F/k$ is a function field with field of constants $k$ if
$F$ is a finite extension of $k(t)$ where $t$ is transcendental over
$k$ and $k$ is algebraically closed in $F$ (that is, $F$ contains no
elements outside of $k$ that are algebraic over $k$).  Define
$\mathbb{P}_F$ to be the set of all $\mathfrak{p}$ such that
$\mathfrak{p}$ is the maximal ideal of some valuation ring of
$F/k$. 

Let $\varphi \in k(x)$ be a rational function of degree $d$.  We write
$\varphi(x) = p(x)/q(x)$, where $p(x), q(x) \in k[x]$, and we let $P(X,Y)$ and $Q(X,Y)$ be the
degree $d$ homogenizations of $p$ and $q$ respectively; that is,
$P(X,Y) = Y^d p(X/Y)$ and $Q(X,Y) = Y^d q(X/Y)$.  We set $P_0 = P$ and
$Q_0 = Q$ and define $P_n$ and $Q_n$ recursively by $P_n(X,Y) =
P(P_{n-1}(X,Y), Q_{n-1}(X,Y))$ and $Q_n(X,Y) = Q(P_{n-1}(X,Y),
Q_{n-1}(X,Y))$ for $n \geq 1$.  Then, defining $p_k = P_k(X,1)$ and $q_k
= Q_k(X,1)$, any root of $\varphi^n(x) -t$ is a root of 
\begin{equation}\label{phi poly}
p_n(x) - tq_n(x),
\end{equation} 
which is a polynomial with coefficients in
$k(t)$.  If $k$ is a number field and $\fp$ is a nonzero prime in its
ring of integers $\fo_k$, we say that the rational function $\varphi(x)$,
defined as above, has good reduction at $\fp$ if all of the
coefficients of $p$ and $q$ have $\fp$-adic absolute value less than
or equal to 1 and for all $\alpha \in {\overline k}$, we have
$\max\{|P(\alpha, 1)|_\fp, |Q(\alpha, 1)|_\fp\} = 1$
and $\max\{|P(1, \alpha)|_\fp, |Q(1, \alpha)|_\fp\} = 1$.

Let $A$ be a Dedekind domain with fraction field $K$, let $P(x)
\in K[x]$, and $L$ be the splitting field of $P(x)$ over
$K$. It is a
standard result that any prime of $A$ that ramifies in the integral
closure of $A$ in $L$ must divide $\Delta(P(x))$, the usual polynomial
discriminant of $P(x)$(see \cite{Jan} or \cite{Lang}, for example). (Here and elsewhere in this paper, if a prime $\mathfrak{p}$ is said to \emph{divide} an element $\alpha$ of $\mathcal{O}_K$, we mean that $v_{\mathfrak{p}}(\alpha)>0$.)  Now
consider the case where $L$ is the splitting field of $\psi(x)-t$ over $k(t)$, where $\psi(x) \in k(x)$. We can write
$\psi(x)=\frac{p(x)}{q(x)}$ for some $p(x), q(x) \in k[x]$ and any prime of $k[t]$ that ramifies in $L$ must divide $\Delta(p(x)-tq(x))$. 
In \cite{CH}, Cullinan and
Hajir show that one may calculate the discriminant in terms of the critical points of $\psi(x)$. 

\begin{lemma} \label{discriminant formula} (\cite[Proposition 1]{CH}) We have 
\begin{align*}
\Delta(p(x)-tq(x)) &= C \res(p'(x)q(x)-p(x)q'(x), p(x)-tq(x))\\
&= C' \prod_{a\in\psi_\mathfrak{c}}(\psi (a)-t)^{e(a/\psi(a))}
\end{align*}
where $C,C'\in k$ are constants, $\psi_\mathfrak{c}=\{a:\psi
'(a)=0\}$, and $e(a/\psi(a))$ is the ramification index of $a$ over
$\psi(a)$.  
\end{lemma}

Thus, we see that any prime $\mathfrak{p}$ of $k[t]$ that ramifies in
a splitting field for $p(x) - t q(x)$  must divide $\prod_{a\in\psi_\mathfrak{c}}(\psi (a)-t)^{e(a/\psi(a))}$.

We now introduce wreath product actions on roots of iterates of
polynomials.  Since we are working with Galois groups that may not be
the full symmetric group, we need slightly more technical definitions
than those of \cite{odoni}.

\begin{definition}
Let $\psi(x),\gamma(x)$ be rational functions in $K(x)$ with $\deg(\psi)=\ell,\deg(\gamma)=d$, such that $\psi(\gamma(x))$ has $\ell d$ distinct roots in $\overline{K}$. A \emph{$\psi, \gamma$-compatible numbering} on the roots of $\psi(\gamma(x))$ is a numbering that assigns to each root a unique ordered pair $(i,j) \in \{1,...,\ell\}\times\{1,...,d\}$ such that if $\alpha_1,...,\alpha_\ell$ are the roots of $\psi$, then the set $\{i\}\times \{1,...,d\}$ is assigned to the roots of $\gamma(x)-\alpha_i$
\end{definition}


\begin{definition} Let $G$ and $H$ be groups acting on the finite sets $\{1,...,\ell\}$ and $\{1,...,d\}$ respectively. We denote the \emph{wreath product} of $G$ by $H$ as $G[H]$, and define it by its action on $\{1,...,\ell\}\times\{1,...,d\}$ as follows. We write $\sigma\in G[H]$ as $(\pi;\tau_1,...,\tau_\ell)$ where $\pi\in G,$ and $ \tau_1,...,\tau_\ell \in H.$ Then $\sigma(i,j)=(\pi(i),\tau_i(j))$. 
\end{definition}

The following lemma generalizes \cite[Lemma 4.1]{odoni}.  

\begin{lemma}\label{wreath product} Let $\psi(x),\gamma(x)\in K(x)$ with $\deg(\psi)=\ell,\deg(\gamma)=d$, $\ell,d \geq 1$, such that $\psi(\gamma(x))$ has $\ell d$ distinct roots in $\overline{K}$. Let $\alpha_1,...,\alpha_\ell$ be the roots of $\psi(x)$, $M_i$ the splitting field of $\gamma(x)-\alpha_i$ over $K(\alpha_i)$, and $G:=\Gal(\psi(x)/K)$. Let $H=\Gal(M_1/K(\alpha_1))$. If $\Gal(M_i/K(\alpha_i))\cong H$ for all $i=1,...,\ell$, then there is an embedding $\iota: \Gal(\psi(\gamma(x))/K) \hookrightarrow G[H]$. Furthermore, there is a $\psi, \gamma$-compatible numbering on the roots such that $\Gal(\psi(\gamma(x))/K) \leq G[H]$. 
\end{lemma} \label{Change 3}

\begin{proof} We may write $$\psi(\gamma(x))=\prod_{i=1}^\ell (\gamma(x)-\alpha_i).$$ We will construct the desired numbering on the roots of $\psi(\gamma(x))$. First choose any numbering $(1,1),...,(1,d)$ on the roots of $\gamma(x)-\alpha_1$. For each $i=2,...,\ell$ choose $\theta_i \in \Gal(\psi(\gamma(x))/K)$ such that $\theta_i(\alpha_1)=\alpha_i$ then $\theta_i(1,j)$ is a root of $\gamma(x)-\alpha_i$ for each $j$. Number the roots of $\gamma(x)-\alpha_i$ so that $\theta_i(1,j)=(i,j)$. 

 Let $\sigma \in \Gal(\psi(\gamma(x))/K)$.  Then $\sigma$ induces a $K$-automorphism $\pi$ that permutes $\{\alpha_1,...,\alpha_\ell\}$. Thus $\pi \in G$. Now fix $i$ and note that $\sigma(i,j)=(\pi(i),s)$ for some $s\in\{1,\dots,d\}$. This defines a map $\tau_i\in\text{Perm}(1,\dots,d)$ by $\tau_i(j)=s$. Then $\sigma(i,j)=(\pi(i),\tau_{i}(j))$ so, using the above wreath product notation, $\sigma=(\pi;\tau_1,...,\tau_\ell) \in G[S_d]$. It remains to show $\tau_i \in H$ for each $i$. Consider $\theta_{\pi(i)}^{-1}\sigma\theta_i (1,j)= \theta_{\pi(i)}^{-1}\sigma (i,j)= \theta_{\pi(i)}^{-1}(\pi(i),\tau_i(j))=(1,\tau_i(j))$. So $\theta_{\pi(i)}^{-1}\sigma\theta_i$ fixes $\alpha_1$ and $\theta_{\pi(i)}^{-1}\sigma\theta_i \big|_{M_1}=\tau_i$ and hence, $\tau_i \in H$.
\end{proof}

\section{Criteria for wreath product} \label{proof of theorem}

Let $\varphi(x) \in k(x)$ be a rational function with degree $d$, such that $\varphi'(x)\neq
0$. Note, the roots of $\varphi^n(x)-t$ are the roots of $p_n(x)-tq_n(x)$, and $p_n(x)-tq_n(x)$ is separable. To see
this, note that since $p_n(x)-tq_n(x)$ is irreducible, if it has
a double root, we must have $p_n'(x)-tq_n'(x)=0$ for all $x$. Then
${(\varphi^n)}'=\frac{q_np_n'-p_nq_n'}{(q_n)^2}=0$ for all $x$. But
since we assumed $\varphi'(x)\neq 0$, ${(\varphi^n)}'(x)\neq 0 $ by
induction on $n$.

Let $K_n$ be the splitting field of $\varphi^n(x)-t$ over $k(t)$,
$E=K_1\cap \overline{k}$, and $G_n:=\Gal(K_n/E(t))$.  We let $G =
G_1$.  We let $\varphi_{\fc}$ denote the critical points of $\varphi$
in $\bP^1({\bar k})$.  
We also adopt some notation regarding extension of primes in finite
extensions of function fields.  Let $L_1 \subseteq L_2$ be a separable
finite extension of function fields.  If $\fp$ is a prime with
discrete valuation ring $\cO_{\fp}$, then we say that the prime $\fq$ of
$L_2$ {\em extends $\fp$ in $L_2/L_1$} if $\fq$ appears in the factorization
of $\fp$ in the integral closure of $\cO_\fp$ in $L_2$. 
(This
terminology is fairly standard.)  Likewise, in the language of points,
we say that a point $\beta \in \bP_{L_2}$ extends a point $\alpha \in
\bP_{L_1}$ in $L_2/L_1$ if the prime ideal corresponding to $\beta$
extends the prime ideal corresponding to $\alpha$ in $L_2/L_1$.  

Our first main theorem gives conditions that ensure that $G_n \cong
[G]^n$.  This is similar to but more general than some recent work of
Pink \cite[Theorem 4.8.1]{Pink1} for quadratic maps, although Pink's criterion is both
  sufficient and necessary, whereas ours is only sufficient.

\begin{theorem}\label{main theorem}
Suppose $\varphi(x) \in k(x)$ is a rational function of degree $d\geq
2$ such that $\varphi'(x)\neq 0$. Fix $N\in\mathbb{N}$ and suppose
there is a subset $S \subseteq \varphi_\fc$
such that the following holds:
\begin{enumerate}
\item  for any $a\in S$, $b \in \varphi_{\fc}$,  and $m,n\leq N$, we have
  $\varphi^m(a)\neq \varphi^n(b)$ unless $a=b$ and $m=n$; and
\item  the group $G$ is generated by the ramification groups of the
  $\varphi(a)$ for $a \in S$, that is 
\[ \big \langle \bigcup_{a \in S} \; \; \bigcup_{\substack{z\text{ extends }\varphi(a) \\
    \text{in }K_1/E(t)}}  I(z/\varphi(a)) \big\rangle = G.  \]
\end{enumerate}
Then we have $G_N \cong [G]^N$.
\end{theorem}

Let $\alpha_1, ..., \alpha_{d^n}$ be the distinct roots of $\varphi^n(x)-t$ in $\overline{k(t)}$. 
Let $M_i$ be the splitting field of $\varphi(x)-\alpha_i$ over $E(\alpha_i) = E(t,\alpha_i)$. Let $\widehat{M_i}:= K_n\left[\prod_{j\neq i}M_j\right]$.

\begin{lemma} \label{wreath subgroup} The group $G_n$ is isomorphic to a subgroup of $[G]^n$.
\end{lemma}

\begin{proof} Note that $E(\alpha_i) \cong E(t)$ so we have $\Gal(M_i/E(\alpha_i)) \cong
  \Gal(K_1/E(t)) \cong G$. Now the result follows immediately from Lemma \ref{wreath product} and induction on $n$.
\end{proof}

To make certain computations easier, we will work with
discriminants in $E[t]$ rather than ramification divisors.  In order
to make this possible, we make a few reductions here.  We note that,
since for any extension $E'$ of $E$, we have $|\Gal(K_N\cdot E' /
E'(t))| < |\Gal(K_N/E(t))|$, it will suffice to show that
$\Gal(K_N\cdot E' / E'(t)) \cong [G]^N$ for some extension $E'$ of
$E$.  Hence, we may assume that $E$ is algebraically closed.  Since
$E$ is then infinite, and a change of variables on $\varphi$ does not
affect $\Gal(K_N/E(t))$, we may therefore assume that 
\begin{equation}\label{inf}
\text{ if $a \in S$ and
$m \leq N$, then $\varphi^m(a)$ is not the point at
infinity.}  
\end{equation}
Furthermore, we may assume that every prime in $E[t]$ is of the form
$(z - t)$ for some $z \in E$, and that the prime at infinity in $E(t)$
does not ramify in $K_n$ for any $n \leq N$.  Hence, in the next two
lemmas, we assume that the conditions of Theorem \ref{main theorem}
hold, that $E$ is algebraically closed, and that \eqref{inf} holds.
 
\begin{lemma}\label{disc}
  Let $n < N$.  The only primes in $E(t)$ that ramify in $K_n$ are those of the form
  $(\varphi^m(a) - t)$ for $a \in \varphi_\fc$ and $m \leq n$.
\end{lemma}
\begin{proof}
  We have seen that the prime at infinity does not ramify in $K_n$.
  For any $i$, we see, by Lemma \ref{discriminant formula}, that the primes
  of $E(t)$ that ramify in $K_n$ are those dividing
\[
\Delta(p_n(x)-tq_n(x))=\prod_{b\in\varphi_\mathfrak{c}}\left(
  (\varphi(b)-t)^{d^{n-1}}(\varphi^2(b)-t)^{d^{n-2}}\dots(\varphi^n(b)-t)
\right)^{e(b/\varphi(b))  } \]
where the above equality follows from repeated application of the
chain rule to iterates of $\varphi$.  
\end{proof}

Before continuing, we make a simple observation.  Let $\alpha_i$ be a
root of $\varphi^n(x) - t = 0$ as above.  Under the inclusion of
fields $E(t) \subseteq E(\alpha_i)$, any prime $(z - \alpha_i)$ extends
the prime $(\varphi^n(z) -  t)$ in $E(\alpha_i)/E(t)$, since $\alpha_i$ is a solution to
$\varphi^n(x) = t$.  

\begin{lemma}\label{disjoint ramification} Let $n < N$ and $a\in S$.  The
  prime $(\varphi(a) - \alpha_i)$ in $E(\alpha_i)$ does not ramify in 
  $\widehat{M_i}$.
\end{lemma}

\begin{proof}
  We will show that
  $(\varphi(a)-\alpha_i)$ does not ramify in $K_n/E(\alpha_i)$ and
  that the primes extending it in $K_n/E(\alpha_i)$ do not ramify in
  $M_jK_n$ if $i \neq j$.

We have assumed that $\varphi^{n+1}(a)-t \neq \varphi^m(b)-t$ for any
$m \leq n$, any $a\in S$, and $b \in \varphi_{\fc}$.  Thus, by Lemma
\ref{disc}, we see that $(\varphi^{n+1}(a) - t)$ does not ramify in
$K_n$. Since $(\varphi(a)-\alpha_i)$ extends $(\varphi^{n+1}(a) - t)$
in $E(\alpha_i)/E(t)$, it follows that $(\varphi(a)-\alpha_i)$ does not
ramify in $K_n$.
  
We can also see that that $(\varphi(a) - \alpha_i)$ does  not ramify in $M_jK_n$ for $j
\neq i$ since the primes of $K_n$ ramifying in $M_jK_n$ are those
dividing
\[\Delta(\varphi(x)-\alpha_j):=\prod_{b\in\varphi_\mathfrak{c}}
(\varphi(b)-\alpha_j)^{e(b/\varphi(b))},\] by Lemma \ref{discriminant formula}.  If a
prime $\fp$ of $K_n$ extending $(\varphi(a)-\alpha_i)$ in
$K_n/E(\alpha_i)$ ramifies in $M_j$ then $\mathfrak{p}$ divides
$\Delta(\varphi(x)-\alpha_j)$, so $\fp|(\varphi(b)-\alpha_j)$ for some
$b \in \varphi_{\fc}$.  Hence, $\fp$ extends the prime $(\varphi(a) -
\alpha_i)$ in $K_n/E(\alpha_i)$ and extends the prime $(\varphi(b) -
\alpha_j)$ in $K_n/E(\alpha_j)$.  Now, the prime $\fp$ extends the prime $(\varphi^{n+1}(a) - t)$ in
$E(\alpha_i)/E(t)$ and extends the prime
$(\varphi^{n+1}(b) - t)$ in $K_n/E(t)$, so we must have
$\varphi^{n+1}(a) = \varphi^{n+1}(b)$ (since $\fp$ can extend exactly
one prime in $K_n/E(t)$).  This means that $a = b$, by condition (1)
of Theorem \ref{main theorem}.  Thus, $\fp$ divides both $(\varphi(a)
- \alpha_i)$ and $(\varphi(a) - \alpha_j)$.  This means that
$(\alpha_i - \alpha_j) \in \fp$.  Since $K_n$ is a splitting field for
$\varphi^n(x) - t$, this implies that $\fp$ ramifies over $\fp \cap
E(t) = (\varphi^{n+1}(a) - t)$, which gives a contradiction, since
$\varphi^{n+1}(a)-t \neq \varphi^m(b)-t$ for any $m \leq n$, any $a\in
S$, and $b \in \varphi_{\fc}$ by \eqref{disc}.
\end{proof}

We now prove Theorem \ref{main theorem}.  
\begin{proof}[Proof of Theorem \ref{main theorem}]
  We will use induction to prove that $G_n \cong [G]^n$ for all $n
  \leq N$.  The case of $n=1$ is clear.  Let $n < N$ and suppose that $G_m \cong
  [G]^m$ for all $m\leq n$; we will show that $G_{n+1} \cong [G]^{n+1}$. First note that
  $E(\alpha_i) \cong E(t)$ so we have $\Gal(M_i/E(\alpha_i)) \cong
  \Gal(K_1/E(t)) \cong G$.

Elements of $\Gal(K_{n+1}/\widehat{M_i})$ and $\Gal(M_i/E(\alpha_i))$ are
determined by their actions on the roots of $\varphi(x)-\alpha_i$.  There
is a natural injective homomorphism from $\Gal(K_{n+1}/\widehat{M_i})$ to
$\Gal(M_i/E(\alpha_i))$ given by restriction of elements of
$\Gal(K_{n+1}/\widehat{M}_i) $ to $M_i$. Let $\Psi:\Gal(K_{n+1}/\widehat{M}_i)
\rightarrow \Gal(M_i/E(\alpha_i))$ be this map.  Let $\mathfrak{p}_1$
be any prime of $M_i$ dividing $\prod_{a \in S} (\varphi(a)-\alpha_i)$,
let $\mathfrak{p}:=\mathfrak{p}_1\cap E[\alpha_i]$, let
$\mathfrak{p}'$ be any extension of $\mathfrak{p}_1$ to $K_{n+1}$ and
let $\mathfrak{p}_2:=\mathfrak{p}'\cap \widehat{M_i}$. Then
$\Psi\big|_{I(\mathfrak{p}'|\mathfrak{p}_2)}
:I(\mathfrak{p}'|\mathfrak{p}_2) \rightarrow
I(\mathfrak{p}_1|\mathfrak{p})$ is an injective homomorphism of the
inertia group of $\mathfrak{p}'$ over $\mathfrak{p}_2$ to the inertia
group of $\mathfrak{p}_1$ over $\mathfrak{p}$. Since $\mathfrak{p}_2$
is unramified over $E(\alpha_i)$ by Lemma \ref{disjoint ramification},
Abhyankar's Lemma (see \cite[III.8.9, page 125]{stichtenoth}) implies
that $e(\mathfrak{p}'|\mathfrak{p})=e(\mathfrak{p}_1|\mathfrak{p})$
and hence
$e(\mathfrak{p}'|\mathfrak{p}_2)=e(\mathfrak{p}_1|\mathfrak{p})$. Thus,
$\left|I(\mathfrak{p}'|\mathfrak{p}_2)\right| =
\left|I(\mathfrak{p}_1|\mathfrak{p})\right|$ and
$\Psi\big|_{I(\mathfrak{p}'|\mathfrak{p}_2)}$ must be an isomorphism.

Consider $I \subseteq \Gal(M_i/E(\alpha_i))$, the subgroup generated by 
$\{I(\mathfrak{q}|\fq \cap E(\alpha_i)): \mathfrak{q} \in \mathbb{P}_{M_i}, \mathfrak{q} |\prod_{a \in S} (\varphi(a)-\alpha_i)\}$, 
and $I' \subseteq \Gal(K_{n+1}/\widehat{M_i})$, 
the subgroup generated by 
$\{I(\mathfrak{q}'|\fq' \cap \widehat{M_i}): \mathfrak{q}' \in \mathbb{P}_{K_{n+1}}, \mathfrak{q}' |\prod_{a \in S} (\varphi(a)-\alpha_i)\}$.  
Then $\Psi |_{I'} :I'\rightarrow I$ is an isomorphism. So
$\Gal(K_{n+1}/\widehat{M_i})$ contains an isomorphic copy of $I$. We
have $I \cong G$ by hypothesis so $\Gal(K_{n+1}/\widehat{M_i})$ contains an isomorphic copy of $G$. We also know that $\Gal(K_{n+1}/\widehat{M_i})$ is isomorphic to a subset of  $G$. It follows that $\Gal(K_{n+1}/\widehat{M_i})\cong G$.

Thus, we have  $$|G_{n+1}|=\prod_{i=0}^{n}\mid G\mid^{d^i} = |
[G]^{n+1} |.$$ By Lemma \ref{wreath subgroup}, $G_{n+1}$ is isomorphic to a
subgroup of $[G]^{n+1}$. Hence, $G_{n+1} \cong [G]^{n+1}$, as
desired.  
\end{proof}

\begin{remark} \label{all critical points}
If $S$ is the set of \emph{all} finite critical points of $\varphi$
then condition (2) of Theorem \ref{main theorem}
follows automatically.  To see this, let $I$ be the
  subgroup of $G$ generated by the ramification groups of all the
  critical points.  Then the fixed field $K_1^I$ is unramified
  everywhere over $E(t)$, so $K_1^I = E(t)$ since $E(t)$ has no
  unramified extensions of degree greater than one, by
  Riemann-Hurwitz. Thus, $I=G$ as desired. We use
this fact in the proof of Theorem~\ref{collide} 
\end{remark}

In Theorem \ref{main theorem}, $E$ was taken to be the algebraic closure of $k$ in $K_1$. In the following proposition we show that algebraic closure of the base field is a necessary condition for the iterated Galois groups to be the full iterated wreath products.

\begin{proposition}\label{not-w}  Let $H = \Gal(K_1/k(t))$.
  Suppose that $k$ is not algebraically closed in $K_1$.  Then
  $\Gal(K_n/k(t))$ is a proper subgroup of $[H]^n$ for $n>1$.
\end{proposition}

\begin{proof} If $k$ is not algebraically closed in $K_1$ then $k(t)$
  is a proper subfield of $E(t)$ so $G$ is a proper subgroup of
  $H$. Thus, we have $|G|<|H|$. Now note that $E \subset K_n$ for
  $n\geq 1$ so $E(\alpha_i) \subset \widehat{M_i}$. Then
  $\Gal(K_{n+1}/\widehat{M_i})$ is isomorphic to a subgroup of 
  $\Gal(M_i/E(\alpha_i))\cong G$. Therefore, $|\Gal(K_{n+1}/K_n)|<|H|^{d^n}$
  and 
\[ |\Gal(K_{n+1}/k(t))|<|[H]^n|,\]
as desired.
\end{proof}

\section{Specializations of Galois groups} \label{special}
Our main results will involve working over Galois extensions of
function fields whose fields of constants are number fields and
reducing modulo primes of the number fields. The notion of
specializing Galois groups is most easily stated in a great deal of
generality, so we work over Noetherian integral domains here, rather
than merely over rings of integers in number fields.  

Throughout out this section, we let $F(D)$ denote the field of fractions
of $D$ for an integral domain $D$.

Let $R$ be a Noetherian integral domain of characteristic 0 and let $A$ be a finitely
generated $R$-algebra that is an integrally closed domain.  Let $h(x) =
\sum_{i=1}^d a_i x^i \in A[x]$ be a
nonconstant polynomial that is irreducible in $F(A)[x]$. Let $B
= A[\theta_1, \dots, \theta_n]$ where $\theta_i$ are the roots of $h$
in some splitting field for $h$ over $F(A)$.  We let $X$ denote $\Spec
A$ and let $Y$ denote $\Spec B$.  For any prime $\fp$ of $R$, we let
$X_\fp$ (resp. $Y_\fp$) denote the fiber $X \times_{\Spec R}
F(R/\fp)$ (resp.  $Y \times_{\Spec R}
F(R/\fp)$).  We let $(0)$ denote the zero ideal in $R$.  Note that
since $R$ is an integral domain, $(0)$ is Zariski dense in $\Spec R$.
In particular, any constructible subset of $\Spec R$ that contains
$(0)$ must be Zariski dense and open.  

Suppose that $F(R)$ is algebraically closed in both $F(A)$ and $F(B)$
(this is a crucial assumption, see Remark \ref{geo}).  Then, since $A$ and
$B$ have characteristic $0$ we see that $X_{(0)}$ and $Y_{(0)}$ are
both  geometrically integral $F(R)$-schemes (see \cite[Proposition
5.5.1]{GW}, for example); in other words, $A \otimes_{F(R)} k'$ and
$B \otimes_{F(R)} k'$ are integral domains for any algebraic extension
$k'$ of $F(R)$.   Hence,
by \cite[9.7.7]{EGAIV}, we see that the set of $\fp \in \Spec R$ such
that $X_\fp$ and $Y_\fp$ are geometrically integral forms a Zariski
dense open subset of $\Spec R$.   Thus, if we let $W_1$ denote the set of $\fp \in \Spec R$ such that
$A/\fp A \otimes_R F(A/\fp)$ and $B/\fp B \otimes_R F(B/\fp)$ are
integral domains, then $W_1$ is a Zariski dense open subset of $\Spec
R$.  

Let $Z_2$ be the set of primes of $A$ that do not contain $a_d$, the
leading coefficient of $h$.  Then $Z_2$ is a Zariski dense open subset
of $\Spec A$.  Let $\pi_{AR}: \Spec A \lra \Spec R$ be the map induced by
the inclusion of $R$ into $A$ and let $W_2 = \pi_{AR}(Z_2)$.   Then by Chevalley's theorem on images of constructible
sets (see \cite[Theorem 10.70]{GW}, for example), $W_2$ must be a
constructible subset of $\Spec R$; since this subset contains the zero
ideal, it must therefore be open and dense.  Likewise, for each $i \not= j$, the set of primes $U_{ij}$
of $\fp$ in $\Spec B$ that do not contain $\theta_i - \theta_j$ form
a Zariski dense open subset of $\Spec B$.  Chevalley's theorem thus
implies that there is a Zariski dense open subset $W_3 \subseteq \Spec
R$ such that for all $\fp \in W_3$ and any $i \not= j$, we have
$r_\fp(\theta_i) \not= r_\fp(\theta_j)$.   Let $W = W_1 \cap W_2 \cap
W_3$.

Now, let $\fp \in W$.  We let $(A)_\fp$ and $(B)_\fp$ denote $A/\fp A
\otimes_R F(A/\fp)$ and $B/\fp B \otimes_R F(B/\fp)$, respectively.
We let $h_\fp$ denote the image of $h \in (A)_\fp[x]$ under the
reduction map from $A$ to $(A)_\fp$.  We let $r_\fp$ denote the
reduction map from $B$ to $(B)_\fp$.  Since $r_\fp$ is a homomorphism
of rings, it is clear that if $\theta_i$ is a root of $h$, then
$r_\fp(\theta_i)$ is root of $h_\fp$; furthermore, $h_\fp$ splits into
distinct linear factors in $F(B/\fp B)[x]$, since $h$ splits into
distinct factors in $B[x]$ and $r_\fp(\theta_i) \not= r_\fp(\theta_j)$
for all $i \not= j$.  Thus, $F( (B)_\fp )$ is a splitting field for
$h_\fp$ over $F( (A)_\fp)$ so $F( (B)_\fp )$ is a Galois extension of
$F( (A)_\fp )$, and we have $[F( (B)_\fp ): F( (A)_\fp )] = \#\Gal
(h_\fp(x) / F( (A)_\fp ))$

Now, given any $\sigma \in \Gal(h(x)/F(A))$, we see that $\sigma: B
\lra B$ since $\sigma$ permutes the $\theta_i$, all of which are in
$B$.  Since $\sigma$ acts identically on $R$, and thus on $\fp$, we
see that $\sigma$ is an automorphism of $R$-algebras and that
$\sigma(\fp B) = \fp B$.  Thus, $\sigma$ induces a homomorphism
$\sigma_\fp: (B)_\fp \lra (B)_\fp$.  If
$\sigma \tau$ is the identity on $B$ for $\tau \in \Gal(F(B)/F(A))$, then clearly $\sigma_\fp
\tau_\fp$ is the identity on $(B)_\fp$, so $\sigma_\fp$ is an
automorphism of $(B)_\fp$.  It extends to an automorphism of
$F((B)_\fp)$, because $(B)_\fp$ is an integral domain.  Thus, we have
a homomorphism
\[ \rho_\fp: \Gal(h(x) / F(A)) \lra \Gal (h_\fp(x) / F((A)_\fp)) \]
with the property that
\[ \rho_\fp(\sigma)(r_\fp(\theta_i)) = r_\fp (\sigma(\theta_i)) \]
for all $\sigma \in \Gal(h(x)/ F(A))$ and all roots $\theta_i$ of $h$
in $B$.

\begin{proposition} \label{EGA}
For all $\fp \in  W$ we have the following:
\begin{enumerate}
\item[(i)] $r_\fp$ induces a bijection between the roots of $h$ and the
  roots of $h_\fp$; and 
\item[(ii)] $\rho_\fp$ is an isomorphism of groups; 
\end{enumerate}

\end{proposition}

\begin{proof}
  Let $\fp \in W$.  
  Since $r_\fp(\theta_i) \not= r_\fp(\theta_j)$ for all $i \not= j$,
  we see that (i) follows immediately.
 
  Let $\sigma$ be a nonidentity element of $\Gal(h(x) / F(A))$.  Then,
  for some $\theta_i$ we have $\sigma(\theta_i) = \theta_j$ for some
  $\theta_j \not= \theta_i$.  Since $r_\fp(\theta_i) \not=
  r_\fp(\theta_j)$ for any $\theta_i \not= \theta_j$, it follows that
  $\rho_\fp(\sigma)(r_\fp(\theta_i)) \not= r_\fp(\theta_i)$, so
  $\rho_\fp(\sigma)$ is not the identity.  Thus, $\rho_\fp$ must be
  injective.
  
  As before, we let $\pi_{AR}: \Spec A \lra \Spec R$ be the map
  induced by the inclusion of $R$ into $A$.  Let $\Spec C$ be an open
  affine subset of $\pi_{AR}^{-1}(W)$ such that $\pi_{AR}(\Spec C)$
  contains $\fp$.  Then, since $h_\fp$ is
  separable (because the $r_\fp(\theta_i)$ are distinct) and $a_d$ is
  a unit in $C$, we have
\begin{equation}\label{deg}
\#\Gal (h_\fp(x) / F( (A)_\fp )) \leq \#\Gal( h(x) / F(A)).
\end{equation} 
by \cite[Lemma 2.4]{odoni}.  It follows that $\rho_\fp$ is surjective
and is therefore an isomorphism
of groups.   
 
\end{proof}

\begin{remark}\label{geo}
    When $F(R)$ is not algebraically
  closed in $F(B)$, many of the arguments in this section do not work.
  For example, if $R = \bZ$, $A= \bZ[t]$, and $B = \bZ[\sqrt[3]{t},
  \xi_3]$, for $\xi_3$ a cube root of unity, then the Galois group of
  $F(B)$ over $F(A)$ has order 6, but when one mods out by a prime $p
  \equiv 1 \pmod{3}$, one does not obtain an integral domain.  This
  explains why $\Gal((x^3 - t) / \bQ)$ can have order 6, even though are
  infinitely many $p$ such that $\Gal((x^3 - t) / \F_p)$ has order 3.
  Note that we still have $\# \Gal((x^3 - t) / \F_p)
  \leq \# \Gal((x^3 - t) / \bQ)$ for all $p \not= 3$, as in \cite[Lemma 2.4]{odoni}.
\end{remark}

\section{The Chebotarev density theorem for function fields}\label{CS}

We begin by showing that if $\fp$ is a prime of good reduction for
$\varphi$, then the number of periodic points for $\varphi_\fp$
is bounded above by $\# \varphi_\fp^n (\bP^1(\F_q))$ for any $n$, where $\F_q$ is the
residue field of $\fp$.   This follows from a very general principle,
which we now prove.  

\begin{definition} Let $T:\cU\rightarrow \cU$ be any map of a set $\cU$ to itself. For $u \in \cU$ define $T^0(u)=u$ and $T^n=T(T^{n-1}(u))$. We say that $u$ is \emph{periodic} if $T^k(u)=u$ for some $k \in \mathbb N$ and we say $u$ is \emph{preperiodic} if $T^k(u)$ is periodic for some $k \in \mathbb Z_{\geq 0} $. We denote the set of periodic points $\Per(T)$ if the set $\cU$ is clear from the context. 
\end{definition}

\begin{lemma}\label{S}  If $\cU$ is finite then every point of $\cU$ is
  preperiodic and $\Per(T)=\cap_{n=0}^\infty T^n(\cU)$. In
  particular, $\#\Per(T) \leq \# T^n(\cU)$ for any positive
  integer $n$.  
\end{lemma}
\begin{proof}
Suppose that $\cU$ is finite and let $u \in \cU$. Then by the pigeonhole principle $\exists m,n\in\mathbb{N}$
such that $T^m(u) = T^n(u)$, so $u$ is preperiodic.

Suppose that $u \in \cU$ is periodic.  Write $T^i(u) = u$ for
some $i > 0$. Then $u \in T^{ik}(\cU)$ for all $k > 0$. Since $T^{in}(\cU)=T^n(T^{n(i-1)})$, we have that $T^{in}(\cU) \subseteq T^{n}(\cU)$, so $u\in T^n(\cU)$ for every $n$.

Suppose that $u\in\cap_{n=0}^\infty T^n(\cU)$. Then we may form a sequence $$\{T(a_1), T^2(a_2), T^3(a_3),\dots\}$$ such that $\forall i,T^i(a_i)=u$. Since $U$ is finite, the pigeonhole principle gives that $\exists i,j$ with $j>i$ such that $a_i=a_j$. Then $u$ is periodic, as $$u=T^{j}(a_j)=T^{j-i}(T^i(a_j))=T^{j-i}(T^i(a_i))=T^{j-i}(u).$$
\end{proof}

In order to apply the Chebotarev density theorem for function fields \cite{murty},
we establish further notation.  Let $L$ be a function field over a
finite field $\F_q$, and let $M$ be a finite extension of $L$. Let
$\alpha$ be a degree one prime in $L$, that is a prime whose residue
field is $\F_q$.  Suppose that $\alpha$ does not ramify in $M$.  Then
for each prime $\gamma$ in $M$ lying over $\alpha$, there is a unique
Frobenius element $\Frob(\gamma/\alpha)$ such that
$\Frob(\gamma/\alpha)$ fixes $\gamma$ and acts as $x \mapsto x^q$ on
the residue field $\ell_\gamma$ of $\gamma$.  We let $\Frob(\alpha)$
denote the conjugacy class of $\Frob(\gamma/\alpha)$ in $\Gal(M/L)$
(note that elements of this conjugacy class correspond to
$\Frob(\gamma'/\alpha)$ as $\gamma'$ ranges over all primes of $M$
lying over $\alpha$).

\begin{proposition}\label{for_use}
  Let $k$ be a number field, let $K = k[t]$, and let
  $\varphi: \bP_k^1\to \bP_k^1$ be a rational function. Let $n\in\mathbb{Z}^+$ and $K_n$ be a splitting field of $\varphi^n(x)
  - t$ over $K$ for some $n$, and let $G_n$ be $\Gal(K_n/K)$.  Suppose that $k$ is algebraically
  closed in $K_n$.  Let $\delta > 0$.  Then there is a constant
  $M_\delta$ such that for all $\fp$ with $\Norm(\fp) > M_\delta$, we have
  \begin{equation}\label{proportion}
  \frac{\#\Per(\varphi_\fp)}{\Norm(\fp) + 1} \leq \FPP(G_n) + \delta.
  \end{equation}
\end{proposition}

\begin{proof} 

  Let $\fp \in \Spec \fo_k$ be a prime of good reduction for $\varphi$
  such that we have $\Gal((K_n)_\fp/(K)_\fp)\cong G_n$
  and let $\F_q$ denote its residue field $\fo_k / \fp$.  We let
  $\varphi_\fp$ denote the reduction of $\varphi$ modulo $\fp$ and let
  $(K_n)_\fp$ denote the splitting field of $\varphi_\fp^n(x) - t$. Let
  $z$ be a root of $\varphi_\fp^n(x) - t$ in $(K_n)_\fp$ and let $\cS$
  denote the conjugates of $z$ in $(K_n)_\fp$.  Then the map $\varphi_\fp:
  \bP_{\F_q}\to \bP_{\F_q}$ is induced by the inclusion of $(K)_\fp$ into $(K)_\fp(z)$.
  Let $A_n$ be the integral closure of $\F_q[t]$ in $(K_n)_\fp$.  Then
  $A_n^{G_n} = \F_q[t]$; that is, $\F_q[t]$ is the set of elements of $A_n$
  that are fixed by every element of $G_n$.  Now, let $(t-\xi)$ be a
  degree one prime in $F_q[t]$ that does not ramify in
  $(K_n)_\fp$, and let $D(\fm / (t-\xi))$ be the decomposition group of a prime
  $\fm$ in $(K_n)_\fp$ that lies over $(t-\xi)$.  Then, by Lemma 3.2 of
  \cite{GTZ}, the number of degree one primes $\beta$ in $(K)_p(z)$
  lying over $(t-\xi)$ is equal to the number of fixed points of
  $D(\fm / (t-\xi))$ acting on $\cS$.  Likewise, working with the
  integral closure $A'_n$ of $\F_q[\frac{1}{t}]$ in $(K_n)_\fp$, we see
  that if $\tau$ is the prime at infinity in $\F_q(t)$ (that is the prime
  $(\frac{1}{t})$ in $\F_q[\frac{1}{t}]$) and $\tau$ does not ramify in
  $(K_n)_\fp$, then the number of degree one primes in $(K)_\fp(z)$ lying
  over $\tau$ is equal to the number of fixed points of $D(\fm /
  \tau)$ acting on $\cS$, where $\fm$ is a prime of $A'_n$ lying over
  $\tau$.  Since decomposition groups over unramified primes are
  generated by Frobenius elements, we see that for any $\alpha
  \in \bP^1(\F_q)$ that does not ramify in $(K_n)_\fp$, we have
\begin{equation}\label{Frob2}
\begin{split}
\text{$\exists\beta \in \bP^1(\F_q)$ such that $\varphi_\fp^n(\beta) = \alpha$}
 \Leftrightarrow
\text{$\Frob(\alpha)$ has a fixed point in $\cS$}.
\end{split}
\end{equation}


Since any dense open subset of $\Spec \fo_k$ contains all but finitely
many primes in $\Spec \fo_k$, it thus follows from Proposition~\ref{EGA}
that for all but finitely many $\fp$, the action of $\Gal( (K_n)_\fp /
(K)_\fp)$ on $\cS$ is isomorphic to the action of $G_n$ on the roots
of $\varphi^n(x) - t$.  For any conjugacy class $C$ of $G_n$, we let
$\psi_C$ denote the number of degree one primes $\alpha$ of
$\bP^1_{\F_q}$ such that $\alpha$ does not ramify in $(K_n)_\fp$ and such
that $\Frob(\alpha) = C$.  Then \cite[Theorem 1]{murty} states that
\begin{equation}\label{MS}
  \left|  \psi_C - (q+1) \frac{\#C}{\#G_n}  \right| \leq  2
  g_{(K_n)_\fp}\frac{\#C}{\#G_n}  + \#R 
\end{equation}
where $g_{(K_n)_\fp}$ is the genus of $(K_n)_\fp$ and $R$ is the set of primes
of $\bP^1_{\F_q}$ that ramify in $(K_n)_\fp$.  Let $\Fix(G_n)$ be the set of
elements of $G_n$ that fix an element of $\cS$.  Then
$\frac{\#\Fix(G_n)}{\#G_n} = \FPP(G_n)$, and for any $\alpha$ outside
of $R$, there is a $\beta$ in $\bP^1(\F_q)$ such that $\varphi_\fp^n(\beta)
= \alpha$ if and only if $\Frob(\alpha) \subseteq \Fix(G_n)$ by
\eqref{Frob2}.  There are at most $\#R$ ramified primes $\alpha$ of $\F_q$ such
that $\alpha \in \varphi_\fp^n(\bP^1(\F_q))$.  Thus,  summing the estimates
in \eqref{MS} over all conjugacy classes in $\Fix(G_n)$ and diving by $q+1$, we then obtain
\begin{equation}\label{again}
  \frac{\varphi_\fp^n(\bP^1(\F_q))}{q+1}  \leq \FPP(G_n) + \frac{\#G_n g_{K_n}}{q + 1} + \frac{2 \#R}{q+1}
\end{equation}
 
The set of primes over which $\varphi_\fp^n$ ramifies has size at most $n (2
\deg \varphi  - 2)$ since $\varphi_\fp$ ramifies over at most $(2 \deg \varphi - 2)$
points and $\varphi_\fp^n$ can only ramify over these points and their first
$n-1$ iterates under $\varphi_\fp$. (Note that $\deg \varphi_\fp = \deg \varphi$
since $\varphi$ has good reduction at $\fp$.)  The size of $G_n$ can be bounded in
terms of $n$ and $d$ only, since it is a subgroup of the symmetric
group on $d^n$ elements. 

Thus, for any $\fp$ of characteristic
greater than $\deg \varphi$ (this guarantees that there is no wild
ramification at for $\varphi_\fp$), we see that $g_{(K_n)_\fp}$ can be bounded
in terms of $\deg \varphi_\fp$ and $n$ by Riemann-Hurwitz; for example,
\[g_{(K_n)_\fp} \leq |G_n| n (2 \deg \varphi_\fp - 2).\]  Hence, by \eqref{again}
there is an $M_\delta$ such that for all $\fp$ with $\Norm(\fp) \geq M_\delta$, we
have
\begin{equation*}
  \frac{\varphi_\fp^n(\bP^1(\F_q))}{q+1}  \leq \FPP(G_n) + \delta.
\end{equation*}
Applying Lemma \ref{S} then finishes our proof.

\end{proof}

We immediately deduce the following as a consequence Proposition \ref{for_use}.  

\begin{corollary}\label{zero}
With notation as in Proposition \ref{for_use}, suppose that $k$ is
algebraically closed in $K_n$ for all $n$.  Then, if $\lim_{n \to
  \infty} \FPP(G_n) = 0$, we have 
\begin{equation}\label{ze} \lim_{\Norm(\fp) \to \infty} \frac{\#\Per(\varphi_\fp)}{\Norm(\fp) +
  1} = 0
\end{equation}
\end{corollary}

\section{Proofs of main theorems}\label{main2}

We will use the following Lemma from \cite{odoni}

\begin{lemma}\label{indi}(\cite[Lemma 4.3]{odoni})
Let $G$ be any transitive group acting faithfully on a finite set
$S$, where $\#S>1$.  Then $\lim_{n \to \infty} \FPP([G]^n) = 0$.  
\end{lemma}

We are now ready to prove our main theorems on proportions of periodic
points.

\begin{proof}[Proof of Theorem \ref{generic}]
Fix $\epsilon > 0$. By Lemma \ref{indi}, there is an $n$ such that $\FPP([S_d]^n) \leq
\epsilon/2$. Let $R = k[a_0, \dots, a_d, b_0, \dots, b_d]$.  Then, the
general rational function
\[ \varphi(x) = \frac{a_d x^d + \dots + a_0}{b_d x^d + \dots + b_0}\]
gives an equation $h_n(x) = \varphi^n(x) - t = 0$.  Let $D$ be the
ring $k[c_0, \dots c_{d-1}]$ and let $\psi: R[t] \lra D$ be the
homomorphism given by $\psi(a_d) =
1$, $\psi(a_i) = c_i$ for $0 \leq i < d$, $\psi(b_0) = 1$, $\psi(b_j)
= 0$ for $1 \leq j \leq d$, and $\psi(t) = 0$.  Then $\psi$ extends to
a map $\psi_1: R[t][x] \lra D[x]$ such that
$\psi_1(h(x)) = x^d + c_{d-1} x^{d-1} + \dots + c_0$.  By \cite[Theorem
1]{odoni}, we have $\Gal(\psi_1(h_n)/F(D)) \cong [S_d]^n$.  Since
Lemma \cite[Lemma 2.4]{odoni} gives
 \[ \#\Gal(\psi_1(h_n)/F(D)) \leq \#\Gal(h_n(x)/F(R[t])), \]
and $\Gal(h_n(x)/F(R[t]))$ is isomorphic to a subgroup of $ [S_d]^n$,
this means that $\Gal(h_n(x)/F(R[t])) \cong [S_d]^n$.   
Proposition \ref{not-w} then tells us that $F(R)$ is integrally closed in the
splitting field of $h_n(x)$, since otherwise $\Gal(h_n(x)/F(R[t]))$
would be a proper subgroup of $[S_d]^n$ for all $n \geq 2$.  Thus, by Proposition \ref{EGA}, if $V_{d,\epsilon}$
is the set of prime ideals $\fm \in \Spec R$ such that $\Gal( (h_n)_\fp
/ (K)_\fp) \cong [S_d]^n$, then $V_{d,\epsilon}$ is Zariski open in $\Spec R$.  Let
$U_{d,\epsilon} = V_{d,\epsilon} \cap \Rat_d$ where $\Rat_d$ is as in the paragraph above the statement
of Theorem \ref{generic}.  
 
Let $\varphi_{ {\vec a}, {\vec b}} \in U_{d,\epsilon}(k)$.  
    Then by Proposition \ref{for_use}, applied to
$\delta = \epsilon/2$, we have 
\[ \frac{\#\Per(\varphi_{{\vec a}, {\vec b}})}{\Norm(\fp) + 1} \leq \epsilon/2 + \epsilon
/2 = \epsilon,\]
for all sufficiently large $\Norm(\fp)$, and our proof is complete.  
\end{proof}

\begin{lemma}\label{extendbase}
Let $k$ be a number field, let $\varphi \in k[x]$, let $\fp$ be a
prime of good reduction for $\varphi$, and let $k'$ be a finite extension
of $k$. Let $\fq$ be a prime of $k'$ such that $\fq \cap \fo_k =
\fp$ and $[ (\fo_{k'}/ \fq)  : (\fo_{k}/ \fp)]  = 1$.  Then $\varphi$
induces a map ${\tilde \varphi}$ over $k'$ such that ${\tilde \varphi}$ has
good reduction at $\fq$  and we have 
$\#\Per({\tilde \varphi}_\fq) = \# \Per(\varphi_\fp)$.    
\end{lemma}
\begin{proof}
We let ${\tilde \varphi}$ be the image of $\varphi$ in $k'(x)$ under the
inclusion $k(x) \subseteq k'(x)$.  Then ${\tilde \varphi}$ has
good reduction at $\fq$.  

Since $[ (\fo_{k'}/ \fq) : (\fo_{k}/ \fp)] = 1$, for any $\beta \in
\fo_{k'}$, there is an $\alpha \in \fo_k$ such that $\beta \equiv
\alpha \pmod{\fq}$.  Thus, there is a natural bijection $\sigma:
\bP^1((\fo_{k'}/ \fq)) \lra \bP^1(\fo_{k}/ \fp)$ such that
$\varphi_\fp(\sigma(z)) = \sigma ({\tilde \varphi}_\fq(z))$ for all $z \in
\bP^1((\fo_{k'}/ \fq))$.  Thus, for each $z \in \bP^1((\fo_{k'}/
\fq))$, we see that $z$ is periodic under $\varphi_\fq$ exactly when
$\sigma(z)$ is periodic under $\varphi_\fp$. Hence, we have $\#
\Per({\tilde \varphi}_\fq) = \# \Per(\varphi_\fp)$.
\end{proof}

\begin{lemma}\label{eb2}
Let $k$ be a number field, let $\varphi \in k[x]$, let $\fp$ be a
prime of good reduction for $\varphi$, let $k'$ be a finite extension
of $k$, and let $\tp$ denote the extension of $\varphi$ to
$\bP^1_{k'}$.  Then
\[ \liminf_{\substack{\Norm(\fp) \to \infty \\ \text{primes $\fp$
    of $k$}}} \frac{ \# \Per (\varphi_\fp)}{\Norm(\fp) +1} \leq
  \limsup_{\substack{\Norm(\fq) \to \infty \\ \text{primes $\fq$
    of $k'$}}} \frac{ \# \Per (\tp_\fq)}{\Norm(\fq) +1}\]
\end{lemma}
\begin{proof}
  There is a positive proportion of primes $\fp$ in $k$ such that $\fp
  \fo_{k'}$ factors as a product of distinct primes $\fq$ such that
  $[ (\fo_{k'} / \fq): (\fo_k / \fp)] = 1$, by the Chebotarev density
  theorem for number fields (see \cite{Cheb, LS}).  Let $\cP$ be the
  set of all such primes at which $\varphi$ has good reduction.  Let
  $\cP'$ be set of primes $\fq$ of $k'$ such that $\fq | \fp$ for some
  $\fp \in \cP$.  Then, by Lemma \ref{extendbase}, we have 
\begin{equation*}
\begin{split}
\liminf_{\substack{\Norm(\fp) \to \infty \\ \text{primes $\fp$
    of $k$}}} \frac{ \# \Per (\varphi_\fp)}{\Norm(\fp) +1} & \leq
\liminf_{\substack{\Norm(\fp) \to \infty \\ \fp \in \cP}} \frac{ \#
  \Per (\varphi_\fp)}{\Norm(\fp) +1}   \\
 & = \limsup_{\substack{\Norm(\fq) \to \infty \\ \fq \in \cP'}}
 \frac{ \# \Per (\tp_\fq)}{\Norm(\fq) +1} \\
& \leq \limsup_{\substack{\Norm(\fq) \to \infty \\ \text{primes $\fq$
    of $k'$}}} \frac{ \# \Per (\tp_\fq)}{\Norm(\fq) +1},
\end{split}
\end{equation*} 
as desired.
\end{proof}

We now prove Theorem \ref{collide}. Recall that $K_n$ denotes the
splitting field of $\varphi^n(x) - t$ over $k(t)$.
 
\begin{proof}[Proof of Theorem \ref{collide}]
  Let $E$ denote the algebraic closure of $k$ in $K_1$ and let $G$ be
  the Galois group $\Gal( (\varphi(x) - t) / E(t) )$.  Let $I$ be the
  subgroup of $G$ generated by the ramification groups of the
  critical points.  Then $I=G$ by Remark \ref{all critical points}.  Thus, we have $\Gal( (\varphi^n(x) - t) / E(t) )
  \cong [G]^n$ for all $n$ by Theorem \ref{main theorem}. Thus, by
  Corollary \ref{zero},
\[ \lim_{\substack{ \Norm(\fq) \to \infty \\ \text{$\fq$ a prime of
      $E$} }} \frac{\#\Per( {\tilde \varphi}_\fq)} {\Norm(\fq) + 1}
= 0,\]
 and Lemma \ref{eb2} then implies (a).  If $k$ is
algebraically closed in $K_1$, then $k=E$ and (b) follows from
Corollary \ref{zero}.  
\end{proof}

\begin{proposition}\label{uni}
Let $k$ be a number field, let $d >1$, and let $f(x) = x^d + c \in k[x]$ have the
property that 0 is not preperiodic.  Then
\begin{enumerate} 
\item[(a)] \[ \liminf_{\fp \to \infty} \frac{\#\Per(f_\fp)}{\Norm(\fp) + 1}
  = 0;\] 
\item[(b)] if $k$ contains a primitive $d$-th root of unity, we have
 \[\lim_{\fp \to \infty} \frac{\#\Per(f_\fp)}{\Norm(\fp) + 1}
  = 0.\]
\end{enumerate}
\end{proposition}
\begin{proof}
Let $k' = k(\xi_d)$ where $\xi_d$ is $d$-th roof unity, and let
${\tilde f}$ the extension of $f$ to $\bP^1_{k'}$.  Then
the splitting field of ${\tilde f(x)} - t$ over $k(t)$ is simply
$k'(t)(\sqrt[d]{t - c})$ which has degree $d$ over $k'(t)$ and ramifies completely over $t-c$; thus,
the Galois group is generated by the ramification group over $t-c$.
Since the critical point 0 is not preperiodic, we see then that the
conditions of Theorem \ref{main theorem} are met for all $N$.  Thus,
for any $n$, we have   $\Gal( ({\tilde f^n
(x)} - t ) / k'(t))
\cong [C_d]^n$ where $C_d$ is the cyclic group of order $d$.  Thus, by
Corollary \ref{zero} and Lemma \ref{indi}, we have
\[ \lim_{\substack{ \Norm(\fq) \to \infty \\ \text{$\fq$ a prime of
      $k'$} }} \frac{\#\Per( {\tilde f}_\fq)} {\Norm(\fq) + 1}
= 0.\]
Since $k' = k$ if $k$ contains a primitive $d$-th root of unity, (b)
follows immediately from Corollary \ref{zero}; likewise, (a) follows
from Corollary \ref{zero} and Lemma \ref{eb2}.  
\end{proof}

\begin{theorem}\label{uni2}
Let $k$ be a number field, let $d >1$, and let $f(x) = x^d + c \in
k[x]$.  Then 
\[ \liminf_{\fp \to \infty} \frac{\#\Per(f_\fp)}{\Norm(\fp) + 1}
  = 0\]
unless $f$ is the Chebyshev polynomial $x^2 - 2$.   
\end{theorem}

\begin{proof}
 If 0 is not
preperiodic under $f$ then the desired result follows immediately from
Proposition \ref{uni}.  

If 0 is preperiodic for $\sigma^{-1}f\sigma$, then $f$ is post-critically finite; that is, every
critical point of $f$ is preperiodic. By \cite[Theorem 1.1]{Jones}, we
must then have $\lim_{n \to \infty} \FPP( \Gal( (f^n(x) - t) / \bC(t) ))
= 0$ unless either (a) $f$ is conjugate to $\pm T_d$, where $T_d$ is
a Chebyshev polynomial of degree $d$ or (b) there is a fixed point
$\alpha \in \bC$ of $f$ such that $f^{-1} (\alpha) \setminus\{ \alpha
\}$ is a nonempty set of critical points of $f$.  It is clear that (b)
cannot happen for maps of the form $x^d + c$, since the inverse image
of any point contains either a single critical point or more than one
point that is not critical.  Furthermore, when $d > 2$, no $\pm T_d$
can be conjugate to $x^d + c$, since the derivative of $\pm T_d$
cannot be a perfect $(d-1)$ power (since, for example, $\pm T_d'$ has
a nonzero term of degree $d-3$ but no term of degree $d-2$).  In the
case where $d=2$, the only conjugate of $\pm T_d$ that has the form $x^2 +
c$ is $x^2 - 2$ (see \cite[Corollary 1.3]{Jones}).   

Now, assume that $f(x) \not= x^2 - 2$.  
Then, from above, we see that if $L_n$ is the splitting field of $f^n(x) - t$ over
$\bC(t)$ and $G_n = \Gal(L_n/ \bC(t))$, then $\FPP(G_n)$ goes to zero as $n$ goes to infinity. Now, let
$K_n$ be the splitting field of $f^n(x) - t$ over $k(t)$ and let $k_n$
be the algebraic closure of $k$ in $K_n$.  Since $K_n$ and $\bC(t)$
are disjoint over $k_n(t)$ we see that every action of $G_n$ on the
roots of $f^n(x) - t$ restricts to a unique action of $\Gal(K_n /
k_n(t))$ on the roots of $f^n(x) - t$.  For each $k_n$, there is a
positive proportion of primes $\fp$ in $k$ such that $\fp \fo_{k_n}$
factors as a product of distinct primes $\fq$ such that $\fo_{k_n} /
\fq = \fo_k / \fp$ by the Chebotarev density theorem for number fields
(see \cite{Cheb, LS}). Let $U_n$ be the set of all such primes.

Choose any $\epsilon> 0$.  Then there is some $n$ such that the
proportion of fixed point elements in $G_n$ is less than $\epsilon/2$.
Then, using Proposition \ref{for_use} with $\delta = \epsilon/2$, we
see that for all sufficiently large $\fq$, the proportion of periodic
points for $f_\fq$ is at most $\epsilon$.  Thus, there is an element of
$\fq \in U_n$ such that the proportion of periodic points for $f_\fq$
is at most $\epsilon$.  Letting $\fp = \fq \cap \fo_k$, the
proportion of periodic points for $f_\fp$ is at at most $\epsilon$, by
Lemma~\ref{extendbase}.  So
we have
\[ \liminf_{\Norm(\fp) \to \infty} \frac{\#\Per(f_\fp)}{\Norm(\fp) + 1} = 0,\]
as desired.

\end{proof}

We can now prove Theorem \ref{quad} quite easily. 

\begin{proof}[Proof of Theorem \ref{quad}]
We choose a linear
polynomial $\sigma = ax + b \in k[x]$ such that $(\sigma^{-1} f \sigma)(x) = x^2 + c$
for some $c \in k$.  Since $\sigma$ is an automorphism of $\fo_k /
\fp$ for all but at most finitely many primes $\fp$ of $\fo_k$, it follows
that for all but at most finitely many primes $\fp$, we have
$\#\Per(f_\fp) = \#\Per ( (\sigma^{-1} f \sigma)_\fp )$.  The result
now follows immediately from Theorem \ref{uni2}.
\end{proof}

\begin{remark}
We note that the techniques Jones \cite{Jones2} uses to control $\Gal(
( (f^n(x) - t) / \bC(t) ))$, for $f$ a post-critically finite
polynomial, are completely different than the wreath
product techniques used here.  Whereas the wreath product techniques here
are mostly algebraic (relying on disjointness of ramification in field
extensions), Jones relies on the complex-analytic theory of iterated monodromy groups  
\end{remark}

\section{Examples} \label{examples}

We end with a discussion of how proportions of periodic points behave
for powering maps, Chebyshev, and Latt\'es maps as we vary over primes
in $\bZ$.  Note that Manes and Thompson \cite{MT} have previously
analyzed periodic points for Chebyshev maps in $\F_{p^n}$ as $n$ goes
to infinity.  In these examples, we provide a mostly elementary
analysis, with no estimates of proportions of fixed-point elements for
iterated Galois groups; for a more Galois theoretic discussion of related
issues, see \cite{Jones}.

\begin{example} \label{powering} Let $f(x) = x^d$.  Let $k$ be a
  number field.  By the Chebotarev density theorem for number fields,
  for any $m$ there are infinitely many primes $\fq$ of $\fo_k$ such
  that $d^m$ divides $q-1$ where $\F_q$ is the residue field $\fo_k /
  \fq$.  For each such prime $f_\fq^m$ is a $d^m$-to-one map on $(\F_q)^*$ so the
  proportion of periodic points is at most $1/d^m + 2 / (q+1)$ (the
  two comes from the fact that 0 and $\infty$ are periodic).  Thus, we
  see that 
\[\liminf_{\Norm(\fq) \to \infty}
  \frac{\#\Per(f_\fq)}{\Norm(\fq) + 1} = 0.\]
\end{example}

\begin{example}\label{Cheby}
  Let $f$ be a Chebyshev polynomial satisfying $f(x + \frac{1}{x}) =
  x^d+ \frac{1}{x^d}$.  Here we work only over $\bQ$.
  We can give an elementary description of the asymptotic behavior of
  the proportion of periodic points for $f_p$; it depends very much on
  whether or not $d$ is a prime power.

Define $\pi(x) = x + \frac{1}{x}$ and $g(x) = x^d$.  Then we have
\[\begin{CD}
     \bP^1  @>g>> \bP^1\\
   @V{\pi}VV          @V{\pi }VV \\
    \bP^1  @>f>> \bP^1\\
\end{CD} \]

Take any $\alpha \in \bP^1(\F_p)$.   Let $\beta\in\bP^1(\F_{p^2})$ such that  $\pi(\beta) = \alpha$.   We
will show that $\beta$ is $g$-periodic if and only if $\alpha$ is
$f$-periodic.   If
$\beta$ is $g$-periodic, then $g^m(\beta) = \beta$ so $f^m(\alpha) =
\alpha$, so $\alpha$ must be $f$-periodic.   Conversely, suppose that
$\alpha$ is $f$-periodic.  If $\beta$ equals 0 or $\infty$, then
$\alpha$ is $\infty$ so both $\alpha$ and $\beta$ are periodic. Suppose $\beta \not= 0, \infty$.  If $f^m(\alpha) = \alpha$ for some $m$, then
$\pi(g^m(\beta)) = \alpha$ for some $m$, so $g^m(\beta) = \beta$ or
$g^m(\beta) = 1 / \beta$. If  $g^m(\beta) = \beta$, then $\beta$ is
obviously $g$-periodic; if $g^m(\beta) = 1/\beta$, then $\beta^{d^m} =
1/\beta$ so $(1/\beta)^{d^m} = \beta$ so $g^{2d}(\beta) = \beta$ so
$\beta$ is still periodic.
 
Let $U$ be the set of $z \in \F^*_{p^2}$ such that $\pi(z)
\in \F_p$.  We see that if $z \in U$ and $z \notin \F_p$, then $z$ and
$1/z$ are the roots of the quadratic polynomial $T^2 - (z + 1/z)T +
1$, so $z$ and $1/z$ are conjugate over $\F_p$.  Hence, we have $z^p =
1/z$ so $z^{p+1} = 1$.  Thus, we see that $U = (\F_p)^* \cup U_{p+1}$ where
$U_{p+1}$ is the set of points in $\F^*_{p^2}$ whose order divides
$p+1$.  The elements of $U$ that are $g$-periodic are simply the ones whose
order is coprime to $d$.

When $d$ is a power of an odd prime, either $p+1$ or $p-1$ is prime to
$d$, so we obtain at least $p-1$ $g$-periodic points.  Since $\pi$ is
two-to-one at all but two points of $U$, we see immediately that $\lim
\inf_{p \to \infty} \#\Per(f_p) / p \geq 1/2$.  Now, there are $p$
such that $p - 1 \equiv 1 \pmod{d^r}$, for any positive integer $r$ by
the Dirichlet theorem for primes in arithmetic progressions (which may
be regarded as a special
case of the Chebotarev density theorem for number fields), so the
proportion of $g$-periodic points in $\F_p^*$ can be made as small as
desired. Thus, we have
\[\liminf_{p \to \infty} \frac{\#\Per(f_p)}{p} =  1/2.\]
 
Suppose that $d$ is a power of 2.  Then at least one of $p-1$ and $p+1$ is not
divisible by 4.  Arguing as in the case of odd prime powers (only with
2 dividing both $p-1$ and $p+1$ for $p > 2$), we see that
$\liminf_{p \to \infty} \#\Per(f_p) / p \geq 1/4$.  For any $r$,  there are
infinitely many $p$ such that $p \equiv 1 \pmod{2^r}$, again by
Dirichlet's theorem on primes in arithmetic progressions.  For such primes,
half of the elements of $U_{p+1}$ are $g$-periodic and at most $1/2^r$
points in $\F_p^*$ are $g$-periodic, so we thus obtain
\[\liminf_{p \to \infty} \frac{\#\Per(f_p)}{p} =  1/4.\]

When $d$ has at least two distinct prime factors $\ell_1$ and
$\ell_2$, things are very different.  For
any $r$, we may find $p$ such that $p \equiv 1 \pmod{\ell_1^r}$ and $p
\equiv -1 \pmod{\ell_2^r}$.  Then the proportion of periodic points in
$\F_p^*$ is at most $1/\ell_1^r$ and the proportion of periodic points
in $U_{p+1}$ is at most  $1/\ell_2^r$.  Hence, we see in this case that 
\[\liminf_{p \to \infty} \frac{\#\Per(f_p)}{p} =  0.\]

\end{example} 

\begin{example}
Let $\ell$ be a prime and let $f(x)$ be a Latt\`es map induced by the multiplication-by-$\ell$ map on
an elliptic curve $E$, say defined over $\bQ$.   We will show that in
many cases, we must have
\[\liminf_{p \to \infty} \frac{\#\Per(f_p)}{p} =  0.\]

The argument here is quite similar to that of Example \ref{Cheby},
though the details are more complicated.   Given a
multiplication-by-$d$ (which we denote as $[d]$)  on an
elliptic curve $E$, we have Latt\`es map.  
\[\begin{CD}
     E  @> [d] >> E\\
   @V{\pi}VV          @V{\pi }VV \\
    \bP^1  @>f>> \bP^1\\
\end{CD} \]
The projection $\pi$ here comes from the inclusion of the fixed field of the
elliptic involution $[-1]$ into the function field of $E$.   When $E$
is in Weierstrass form $y^2 = g(x)$, we have simply  $\pi(x,y) = x$.  

We now assume that $d=\ell$ is a prime; letting
$\Gal(\overline{\bQ}/\bQ)$ act on the Tate module $T_{\ell}(E)$ we obtain
a homomorphism $\rho_{\ell} : \Gal(\overline{\bQ}/\bQ) \to
GL_{2}(\bZ_\ell)$.  We further assume that $\ell$ is chosen so that $\rho_\ell$
surjects onto $GL_{2}(\bZ_\ell)$; for $E$ fixed (and without complex
multiplication) this holds for all but finitely many prime $\ell$ by
Serre's celebrated open image theorem (see \cite{SerreOpen}).  

Given a prime $p$, let $F_{p} = \rho_\ell(\Frob_{p})$ denote the image of
the Frobenius conjugacy class $\Frob_{p}$ in 
$GL_{2}(\bZ_\ell)$.  Given $k \in \bZ^{+}$, the Chebotarev
density theorem together with the surjectivity of $\rho_\ell$ implies 
that we have
$$
\sigma_p \equiv \begin{pmatrix} 1 & 0 \\ 0 & -1  \end{pmatrix}
\mod \ell^{k} 
$$
for some $\sigma_p \in  \rho_\ell(\Frob_{p})$
for a positive proportion of primes $p$.  For such $p$, the group $E(\F_p),$
viewed as an abelian group, contains a subgroup $H_{1} \simeq
\bZ/\ell^{k}\bZ$ on which the induced Frobenius action is trivial.
Furthermore, since $\sigma_{p}^{2}$ is congruent to the identity
matrix modulo $\ell^{k}$, there exists a subgroup 
$H_{2} \simeq \bZ/\ell^{k}\bZ$ contained in $E(\F_{p^2})$ such that the
induced Frobenius action on $H_{2}$ is given by multiplication by $-1$.
In order to analyze the action of $f_{p}$ on $\bP^1(\F_p)$, let
$$
S_{1} = \{ x \in \F_p : \text{ $x^{3}+ax+b$ is a quadratic residue mod $p$} \}
$$
and let $S_{2}  = \F_p \setminus S_{1}$  denote the complement.  

We begin with the $f_{p}$-periodic points in $S_{1}$.   With
$G_{1}$ denoting the group of $\F_p$-points on $E$, we note
that $\pi^{-1}(S_{1}) \cup \infty = G_{1}$.  
By
\cite[Proposition~6.52]{silverman-arithmetic-dynamics-book})
$\Per_{n}(f)
= \pi(E[\ell^n-1]) \cup \pi(E[\ell^n+1])$
hence it
is enough to show that $G_{1}$ has small intersection with the union
of $ E[\ell^n\pm 1]$ for all $n$.  When we do this, let $H_{1}'$ denote the
maximal cyclic group of order $\ell^{k_{1}}$ such that 
%
$H_{1}' \subset G_{1}$; we then have $k_{1} \geq k$ and
we also note that there is a projection $G_{1}
\twoheadrightarrow H_{1}'$. Moreover, since $\ell$ is coprime to
$\ell^{n}\pm 1$ for all $n$, the intersection
$$
G_{1} \cap (\cup_{n \geq 1} (\pi(E[\ell^n-1]) \cup \pi(E[\ell^n+1])))
$$
is contained in the kernel of $G_{1} \twoheadrightarrow H_{1}'$.
Consequently the proportion of $f_{p}$-periodic point in $S_{1}$ is at
most $(1+o(1))/\ell^{k}$, as $p \to \infty$.

We next consider the proportion of $f_{p}$-periodic points in $S_{2}$.
Since
$\pi^{-1}(S_{2})$ is contained in the subgroup 
$$
G_{2} :=   \{ P \in E(\F_{p^2}) : P+\Frob_p(P) = 0 \}
$$
(if $x \in S_{2}$ and $y^{2} = x^{3} +ax+b$ then $\Frob_p(y)=-y$) and
$\pi^{-1}(S_{2}) \cup E[2](\F_{p^2}) = G_{2}$ (i.e., they have
essentially the same cardinality) we may argue as before by bounding
the intersection $G_{2} \cap (\cup_{n \geq 1} (\pi(E[\ell^n-1]) \cup
\pi(E[\ell^n+1])))$.  With $H_{2}'$ denoting the maximal cyclic group of
order $\ell^{k_{2}}$ such that $H_{2}' \subset G_{2}$; we
again have $k_{2} \geq k$ and a projection $G_{2}
\twoheadrightarrow H_{2}'$.

Arguing as before we find that the proportion of $f_{p}$-periodic point in
$S_{2}$ is at most $(1+o(1))/\ell^{k}$, as $p \to \infty$.

By the Weil bounds $\# S_{1} = p/2 + O(\sqrt{p})$, hence $\# S_{2} = p/2
+ O(\sqrt{p})$, and we find that the proportion of $f_{p}$-periodic points
$x \in \F_p$ is at most $(1+o(1))/\ell^{k}$, as $p \to \infty$.  
Since $k$ might be taken arbitrarily large, we find that
\[\liminf_{p \to \infty} \frac{\#\Per(f_p)}{p+1} =  0.\]

We end by remarking that $\rho_\ell$ being surjective is a much stronger
assumption than needed --- we only require that the image contains
a sequence of elements $g_{i} \to h_{i}$ (in the $\ell$-adic norm), where
each $h_{i} \in GL_{2}(\bZ_\ell)$ is $\bZ_{\ell}$-conjugate to $J
:= \begin{pmatrix} 1 & 0 \\ 
  0 & -1 \end{pmatrix}$.
For instance, if $\ell$ is odd and $\#(E[\ell] \cap
E(\bQ)) = \ell$, the image of $\rho_\ell$ is much smaller than
$GL_{2}(\bZ_{\ell})$, but it still contains an element $M \in
GL_{2}(\bZ_{\ell})$ which, modulo $l$ is conjugate to $J$. (Since there
is $\ell$-torsion defined 
over $\bQ$, the reduction modulo $\ell$ fixes an $\F_\ell$-line, and the
composition of $\rho_\ell$ with the determinant  surjects onto
$\bZ_{\ell}^{\times})$.  Now, $M$ being conjugate  (modulo $\ell$) to a
diagonal matrix whose eigenvalues are distinct modulo $\ell$ implies a
$\bZ_\ell$-conjugacy $M \sim M' = \begin{pmatrix} \lambda_{1} & 0 \\ 0 &
  \lambda_{2} \end{pmatrix}$ where $\lambda_{1} \equiv 1 \mod \ell$ and
$\lambda_{2} \equiv -1 \mod \ell$.  Since $M$ is in the image, so is
$M^{\ell^{k}}$, and we clearly have $(M')^{l^{k}} \to J$ as $k\to \infty$ (in
the $\ell$-adic metric).

\end{example}

\newcommand{\etalchar}[1]{$^{#1}$}
\providecommand{\bysame}{\leavevmode\hbox to3em{\hrulefill}\thinspace}
\providecommand{\MR}{\relax\ifhmode\unskip\space\fi MR }
\providecommand{\MRhref}[2]{%
  \href{http://www.ams.org/mathscinet-getitem?mr=#1}{#2}
}
\providecommand{\href}[2]{#2}

\end{document}